\documentclass[letterpaper]{amsart}

\usepackage{amsfonts,amssymb,amsthm}
\usepackage{graphicx}
\usepackage{xcolor}
\usepackage[outercaption,ragged]{sidecap}

\newtheorem{thm}{Theorem}[section]

\title{Nonlinear waves and polarization in diffusive directed particle flow}

\author{Heinrich Freist\"uhler \and Jan Fuhrmann}\thanks{Universit\"at Konstanz, Fachbereich Mathematik und Statistik, Universit\"atsstra\ss e 10, 78464 Konstanz, Germany\\Johannes-Gutenberg-Universit\"at Mainz, Institut f\"ur Mathematik, Staudingerweg 9, 55128 Mainz, Germany, fuhrmann@uni-mainz.de}

\begin{document}

\begin{abstract}
We consider a system of two reaction-diffusion-advection equations describing the one 
dimensional directed motion of particles with superimposed diffusion and mutual alignment. 
For this system we show the existence of traveling wave solutions for weak diffusion by 
singular perturbation techniques and provide evidence for their existence also for stronger 
diffusion. We discuss different types of wave fronts and their composition to more complex 
patterns and illustrate their emergence from generic initial data by simulations. 
We also investigate the dependence of the wave velocities on the model parameters.
\end{abstract}

\maketitle

\section{Introduction} \label{sec:Intro}

Traveling waves in reaction-diffusion systems have been investigated for a long time, and applications to 
biological contexts are abundant in the literature (cf. \cite{TysKee1988}, \cite{VolPet2009}, and references therein). 
However, our principal motivating biological system, the cytoskeleton of motile cells, is not 
adequately described by pure reaction-diffusion 
equations since the advective effects of actin flow play a crucial role. A comprehensive reaction-diffusion-advection model for the actin cytoskeleton of a motile cell 
was established in \cite{FuhKasSte2007} and analyzed in \cite{FuhSte2015}.
In the present paper, we investigate a simpler, prototypical reduced system that 
isolates features which seem essential for hyperbolic-parabolic systems with reaction such as the cell motion model of \cite{FuhKasSte2007},
notably regarding the formation, propagation, and interaction of waves.

To obtain the `reduced model' and enable both analytical access and more reliable 
simulation, we have on the one hand chosen a comparatively low overall complexity and on the other hand 
explicitly introduced  features that, from the perspective of the cell motion problem, 
correspond to the (reasonable) assumptions that the filaments are short and capable 
of aligning each other.
The resulting system resembles previous models for directed diffusive 
particle flow (e.g.,\cite{LutSte2002} without diffusion, \cite{Man2010}, more recently \cite{SanPer2016}, \cite{ZhaJin2017}), and has a cross advection structure making it, in our opinion, worthy of attention in 
its own right. We shall see that it exhibits a surprisingly rich dynamic behavior, which renders it
interesting also from an intrinsically mathematical point of view.

We employ geometric singular perturbation theory (established in \cite{Fen1979}, \cite{Szm1991}, 
summarized in \cite{Kue2015}) to find traveling wave solutions to this model. In section \ref{sec:polarization}, the traveling wave problem is formulated and for slow diffusion 
of filaments is observed to be a singular perturbation of the purely hyperbolic limit problem. 
The existence of `polarization waves' for small diffusion is then deduced from the existence of 
traveling fronts in the hyperbolic limit using geometric singular perturbation theory. Possible velocities for monotone polarization waves are discussed by linearization about the asymptotic states. 
We also find `inversion waves'. These do not exist in the hyperbolic limit system, but are 
a distinguishing feature of the full problem with diffusion. Their emergence and properties 
are discussed in section \ref{sec:inversion}. In section \ref{sec:simulations} we present 
wave patterns emerging in simulations of the full PDE system. These are composed of several 
simple waves discussed in sections \ref{sec:polarization} and \ref{sec:inversion}, and we 
investigate the dependence of the wave speeds on the model parameters. Finally, we summarize 
the results and relate them to experimentally observed actin waves in \ref{sec:conclusion}. 
We also comment on the interpretation of the traveling wave solutions in connection with
shock-like waves found in \cite{FuhSte2015}.

\section{Formulation of the model}\label{sec:formulation}

With the goal of understanding the motion of actin filaments in a  model for the cytoskeleton 
proposed in \cite{FuhKasSte2007}, we consider a collection of particles moving in one space dimension with fixed velocity $v$ either to the left or to the right. Denoting their densities by $u_{r}$ and $u_{l}$, for right and left moving filaments, respectively, this reads 
\begin{subequations} 
\begin{eqnarray}\label{eq:filaments_pure}
\partial_{t} u_{r}(t,x) + v\partial_{x} u_{r}(t,x) &=& 0\\
\partial_{t} u_{l}(t,x) - v\partial_{x} u_{l}(t,x) &=& 0.
\end{eqnarray}
\end{subequations}
These equations are strongly simplified versions of the equations describing the densities of actin filaments as established in \cite{FuhKasSte2007}. This cytoskeleton model comprised of hyperbolic and parabolic equations was observed to exhibit shock-like solutions in \cite{FuhSte2015}. As these came as a surprise we decoupled the hyperbolic equations for the motion of filaments from the polymerization dynamics and thus arrived at \eqref{eq:filaments_pure}. 

Rewriting this system in terms of total particle density $u := u_{r}+u_{l}$ and the difference $w := u_{r}-u_{l}$ of right and left oriented particles -- we will call $w$ the polarization -- we can rewrite this into
\begin{eqnarray*}
 \partial_{t}u+\partial_{x}w = 0 &~~~~~~~& \partial_{t}w+\partial_{x}u = 0 
\end{eqnarray*}
where we have chosen the velocity to be $1$ by implicitly rescaling space and time.

As one of the assumptions used on the way from the original cytoskeleton model to \eqref{eq:filaments_pure} was the shortness of the filaments it makes sense to assume them to undergo slow diffusion and being capable of aligning one another:
\begin{subequations}\label{eq:full}
\begin{eqnarray} 
 \partial_{t}u+\partial_{x}w &=& \varepsilon \partial_{xx}u \\
 \partial_{t}w+\partial_{x}u &=& \varepsilon \partial_{xx}w + f(u,w) 
\end{eqnarray}
\end{subequations}
where the slowness of the diffusion is reflected by assuming the diffusion coefficient $\varepsilon$ to be small. The alignment term $f$ accounts for the ability of filaments to turn around those which come from the opposite direction. We shall assume that the majority will be able to turn around the minority more effectively than vice versa and thus $f(u,w) \geq 0$ if $0\leq w\leq u$ and $f(u,w) \leq 0$ if $0 \geq w \geq -u$. It should be noted that the only nonlinear term in \eqref{eq:full} is the alignment term while both diffusion and cross advection are purely linear.

Given a fixed total particle density, the dependence of the alignment term on the polarization will take the form of the force derived from a bistable potential with stable equilibria at $w = \pm u$ and an unstable equilibrium at $w = 0$ as exemplarily given by
\begin{equation}
 f(u,w) ~=~ \alpha f_{0}(u,w) ~:=~ \alpha w \left( 1 - \frac{w^{2}}{u^{2}} \right) \exp\left[-\beta^{2} u^{2} \right]   \label{eq:fsub}
\end{equation}
where the exponential term accounts for a crowding effect which makes the alignment increasingly more difficult when the density becomes too large. A more detailed derivation of system \eqref{eq:full} and possible other alignment terms $f$ is found in \cite{FreFuhSte2014}.

The model contains the two parameters $\varepsilon$ and $\alpha$ denoting the strengths of diffusion and alignment respectively. By rescaling both time and space by $\alpha$,
$$\tilde t = \alpha t,~~~~~~~~~~~~\tilde x = \alpha x$$
we arrive at the rescaled problem
\begin{subequations} \label{eq:rescaled_full}
\begin{eqnarray}
 \partial_{\tilde t} \tilde u + \partial_{\tilde x} \tilde w &=& \alpha\varepsilon\partial_{\tilde x\tilde x} \tilde u \\
 \partial_{\tilde t} \tilde w + \partial_{\tilde x} \tilde u &=& \alpha\varepsilon\partial_{\tilde x\tilde x} \tilde w + f_{0}(\tilde u,\tilde w)
\end{eqnarray}
\end{subequations}
depending only on the single parameter $a:=\alpha\varepsilon$ as rescaled diffusion coefficient. In what follows we shall work with this rescaled reaction-diffusion-advection system and omit the tilde over the variables.

\section{Polarization waves} \label{sec:polarization}

\subsection{Existence of polarization waves}

Looking for traveling wave solutions to  problem \eqref{eq:rescaled_full} we describe putative wave profiles at wave speed $c$ by 
$$U(\xi) = u(t,x)~~ \mbox{ and }~~W(\xi) = w(t,x)~~~~~~\mbox{ with }\xi=x-ct$$ 
and obtain the following system of ordinary differential equations
\begin{subequations}   \label{eq:TWfull}
 \begin{align}
 -c\, U'(\xi) + W'(\xi) &~= a\, U''(\xi)\\   
 -c\, W'(\xi) + U'(\xi) &~= a\, W''(\xi) + f_{0}(U,W)
  \end{align}
\end{subequations}
as traveling wave problem.

As a toy model let us start with the purely hyperbolic system without diffusion (that is, $a=0$), which via $W' = c U'$  can be written as
\begin{equation} \label{eq:TW_hyp}
 U' = \frac{1}{1-c^2} f_0(U,W), ~~~~~~~~W' =  \frac{c}{1-c^2} f_0(U,W).
\end{equation}
That immediately shows that the wave speed $c$ should be different from the intrinsic velocity $\pm 1$ of the particles. As the set of equilibria consists of three rays in the $W$-$U$-plane and all trajectories are straight lines with slope $c^{-1}$ we can easily sketch possible orbits corresponding to the wave fronts (cf. Figure \ref{fig:Phasespace_hyp}) and also calculate the relative values of the asymptotic states. 

\begin{SCfigure}[2][hbt]
 \includegraphics[width=50mm, keepaspectratio]{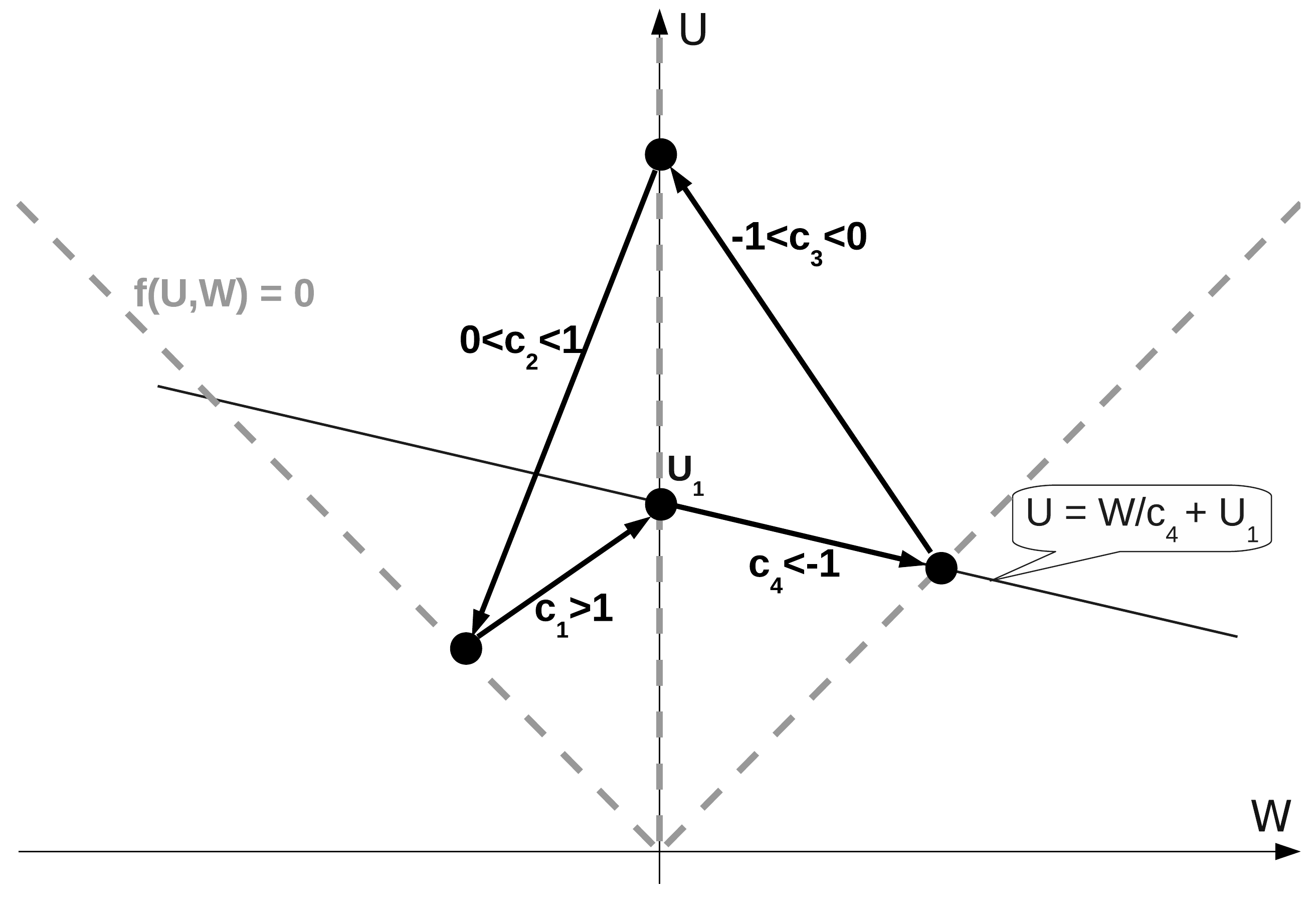}
\caption{Sketch of the equilibria (\textit{gray} lines) for the hyperbolic limit system ($a=0$) and possible heteroclinic orbits connecting them in the upper $W$-$U$ half plane. Note that going through all of the sketched orbits we obtain a pattern of waves with the same asymptotic state as $\xi\!\to\! -\infty$ and $\xi\!\to\!\infty$, respectively.}
\label{fig:Phasespace_hyp}
\end{SCfigure}

More precisely, whenever the wave velocity is non-zero and different from the particle velocity, $c\neq \pm 1$, we conclude the existence of orbits connecting the central equilibrium $W=0$ with any of the outer equilibria $W=\pm U$ . If $\vert c \vert$ approaches one, the equation tells us that apart from the equilibria the rate of change $W'$ of the polarization becomes very large so that we can only expect shocks, or rather contact discontinuities (see, e.g., \cite{Smo1994}), to move with speed $\pm 1$. 

Introducing now a small diffusion coefficient amounts to a singular perturbation to this simple problem which we shall check to be the normally hyperbolic limit of the full problem as $a=\alpha\varepsilon$ decreases to zero. To see this we observe that for the full problem we can write the traveling wave equations as first order system
\begin{subequations} \label{eq:TWfull_system}
 \begin{eqnarray}
  U'(\xi) &~:=~& Z(\xi) \\
  W'(\xi) &~:=~& V(\xi) \\
  a\,Z'(\xi) &~=~& V(\xi) - c\, Z(\xi)\\
  a\,V'(\xi) &~=~& Z(\xi) - c\,V(\xi) - f_{0}(U(\xi),\, W(\xi)).
 \end{eqnarray}
\end{subequations}
which is defined in the half space $H = \{U>0\}\subset \mathbb{R}^4$ and for all $a\in\mathbb{R}$ though we only consider $a\geq0$. The physically meaningful region is the invariant domain
\begin{equation*}
 \mathcal{M} = \left\{ (U,W,Z,V)\in \mathbb{R}^{4} ~\mid~ U>0,\,\vert W\vert < U\right\}\subset H.
\end{equation*}

We should emphasize that in what follows we view \eqref{eq:TWfull_system} as autonomous dynamical system with the traveling wave variable $\xi$ playing the role of time (not to be confused with physical time). This allows us to use the language from \cite{Kue2015} dealing with singular perturbation problems in terms of multiple time scales. The terms \emph{fast}, \emph{slow} and others we will use are thus borrowed from the theory of dynamical systems and refer to $\xi$ as ``time''.

As reduced system (or slow subsystem) for the fast-slow system \eqref{eq:TWfull_system} we recover the traveling wave equation \eqref{eq:TW_hyp} for the hyperbolic system (with $a=0$) acting on the critical manifold
\begin{equation*}
 \mathcal{S} = \left\{ (U,W,Z,V)\in H ~\mid~  Z = \frac{1}{1-c^{2}} f_{0}(U,W), \, V=\frac{c}{1-c^{2}}f_{0}(U,W) \equiv c Z  \right\}.
\end{equation*}
Note that $\mathcal{S}$ is parametrized as graph of the function
$$h:\{(U,W)\subset \mathbb{R}^2 \mid U>0\} \to \mathbb{R}^2,~~~ (U,W) \mapsto \frac{f_0(U,W)}{1-c^2} (1,c)$$
and thus equation \eqref{eq:TW_hyp} makes sense on $\mathcal{S}$ by applying it to the arguments of $h$.

Rewriting \eqref{eq:TWfull_system} in the new fast time variable $\tau = a^{-1}\xi$ yields the fast system
\begin{equation}
 \begin{array}{rclrr}
  U'(\tau) &~=~& a\, Z(\tau) &~~& =: G_1(U,W,Z,V)\\
  W'(\tau) &~=~& a\, V(\tau) && =: G_2(U,W,Z,V)\\
  Z'(\tau) &~=~& V(\tau) - c\, Z(\tau) && =: F_1(U,W,Z,V) \\
  V'(\tau) &~=~& Z(\tau) - c\,V(\tau) - f_{0}(U(\tau),\, W(\tau)) && =: F_2(U,W,Z,V)
 \end{array}
\end{equation}
 which is reduced to a two dimensional, linear, constant coefficient layer problem (fast subsystem) upon setting $a=0$ and prescribing any constant $U=\bar{U}>0$ and $W = \bar{W}$. The eigenvalues of this subsystem's coefficient matrix 
 \begin{equation}
 \frac{\partial (F_1,F_2)}{\partial(Z,V)} =  \begin{pmatrix}
   -c & 1 \\ 1 & -c
  \end{pmatrix}
 ~~~~~~\mbox{ are }~~~~~~ -c\pm 1,
 \end{equation}
which means that as $\tau\to\infty$, the solution to the layer problem approaches the unique equilibrium
\begin{equation} 
Z=Z_\infty=\frac{\bar f}{1-c^2}, V = V_\infty = \frac{c \bar f}{1-c^2}
\end{equation}
whenever $c>1$. These values are $0$ if and only if $\bar f := f_0(\bar U,\bar W)=0$, which reminds us of the fact that the zeros of $f_0$ are equilibria of the full system \eqref{eq:TWfull}. Moreover, both eigenvalues are real and different from $0$ whenever $c\neq \pm1$, and since they are independent of $(\bar W, \bar U)$, we conclude that the invariant manifold $S$ is normally hyperbolic. More precisely, it is attracting if $c>1$, repelling if $c<-1$, and of saddle type if $\vert c\vert<1$. 

 
 \begin{figure}[hbt]
  \includegraphics[width=\linewidth]{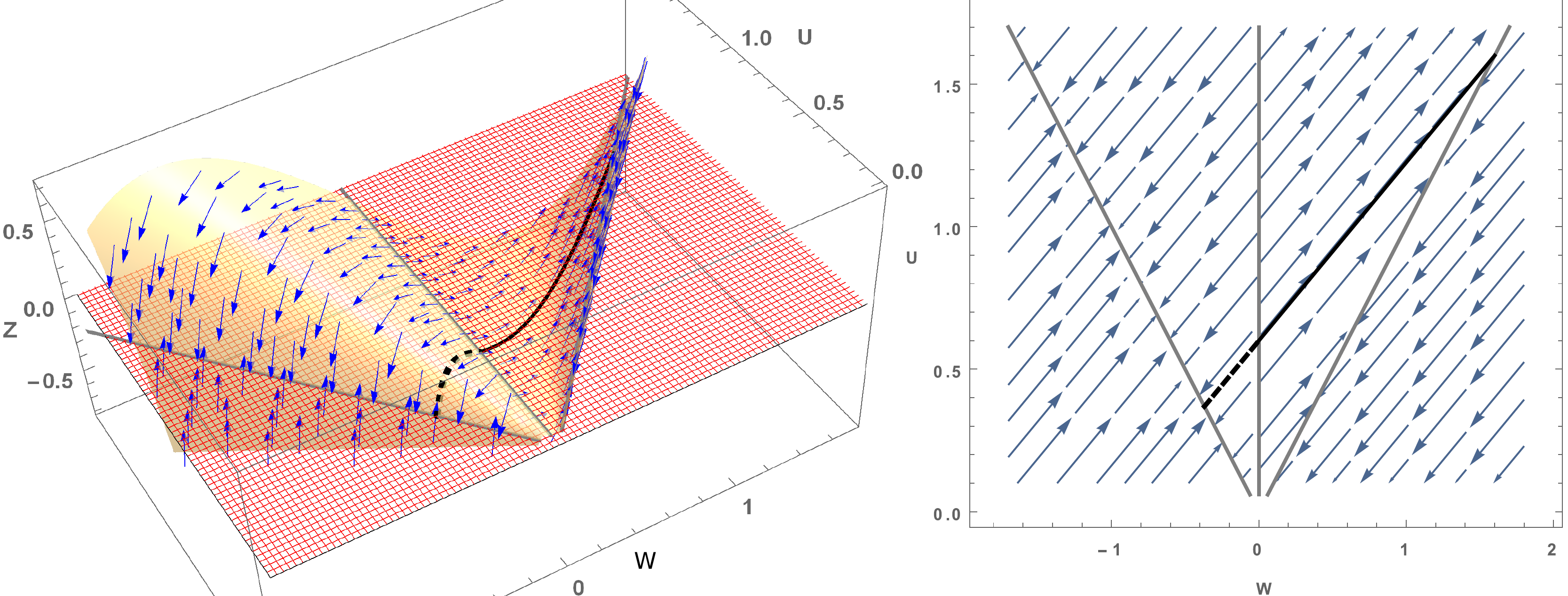}
  \caption{\emph{Right:} Slow flow on the critical manifold in $W$-$U$-$Z$-space (the $V$ direction being suppressed) for $c=1.8$. Two sample heteroclinic orbits from the outer equilibria $W=\pm U$ to the central equilibrium $W=0$ (at $U_1 = 0.6$) and the $Z=0$-plane are shown as well. \emph{Left:} The same flow and trajectories projected to the $W$-$U$-plane (compare Figure \ref{fig:Phasespace_hyp}).}
  \label{fig:slowflow}
 \end{figure}

From the projection of the slow flow on the $W$-$U$-plane it is obvious that given $c\notin\{0,\pm1\}$, there are one-parameter families of heteroclinic orbits for the slow flow which foliate the critical manifold $\mathcal{S}$. As Figure \ref{fig:slowflow} makes apparent, these orbits lie in the intersection of the critical manifold with one of the invariant hyperplanes $aZ+cU-W = C_1$, and the condition $c\notin\{0,\pm1\}$ ensures that the lines of equilibria are transversal to these hyperplanes. Rather than by $C_1$, the families of orbits are best parametrized by the $U$-value $U_1$ of their end point $(U\!=\!U_1,W\!=\!Z\!=\!V\!=\!0)$ located on the $U$-axis, and each of these trajectories is given as the curve $\xi\mapsto(U(\xi),W(\xi),h(U(\xi),W(\xi)))$ with $\xi\mapsto(U(\xi),W(\xi))$ being one of the heteroclinic orbits for \eqref{eq:TW_hyp} (see Figure \ref{fig:Phasespace_hyp}).

The geometric singular perturbation theory developed by Fenichel in \cite{Fen1979}, refined by Szmolyan in \cite{Szm1991} (and exquisitely presented in \cite{Kue2015}) now ensures the existence of similar (families of) trajectories for the full model \eqref{eq:TWfull} if the diffusion coefficient is sufficiently small, where the necessary smallness may depend on the wave speed and the alignment term. We can therefore formulate the following theorem.

\begin{thm}
 Fix a wave speed $c$ different from $\pm 1$ and $0$ and a value $U_1 > 0$. Then, there exists $a_{0}(c)>0$ such that for any $a\in (0,\,a_{0}(c))$ the system (\ref{eq:rescaled_full}) admits traveling wave solutions with velocity $c$ connecting the steady states
\begin{itemize}
 \item[$(i)$]{$W=0$ to $W=U$ for $c<-1$}
 \item[$(ii)$]{$W=U$ to $W=0$ for $-1<c<0$}
 \item[$(iii)$]{$W=0$ to $W=-U$ for $0<c<1$ and}
 \item[$(iv)$]{$W=-U$ to $W=0$ for $c>1$}
\end{itemize}
where the steady state $W=0$ is to be understood as $(U,W) = (U_1,0)$.
\label{thm:TWexist}
\end{thm}

\begin{proof}
 The hyperbolicity of $\mathcal{S}$ and the heteroclinic orbits for the slow flow have already been discussed. The only point to be resolved is the fact that $\mathcal{S}$ is in fact not a smooth manifold due to the singularity at the origin. However, having fixed any value $U_1$ and a velocity $c$ we only require a compact subset 
 \begin{equation*}
  \mathcal{C} := \left\{(U,W,Z,V)\in \mathcal{S} \mid U\in[m,M], -U-\delta < W < U+\delta \right\}
 \end{equation*}
with appropriately chosen positive constants $m$, $M$, and $\delta$ such that the equilibria to be connected by the slow flow lie in the interior of $\mathcal{C}$. 

The heteroclinic orbits of system \eqref{eq:TWfull_system} obtained as perturbations of those for the slow flow on $\mathcal{S}$ then correspond to traveling wave solutions of \eqref{eq:rescaled_full}.
\end{proof}

We shall denote the traveling waves obtained in Theorem \ref{thm:TWexist} connecting a totally polarized state $W=\pm U$ and the non-polarized state $W=0$ as \emph{(de)polarization waves}. Fast polarization waves with $\vert c\vert >1$, traveling faster than  the individual particles, are characterized by a wave profile running through the medium and converting the symmetric, non-polarized state to a fully polarized state. To the contrary, the slow depolarization waves have a velocity smaller than the particle velocity and convert a fully polarized state into the symmetric state $W=0$. 

Returning to the fast subsystem
 \begin{align}
  Z' & = V- cZ,&~~&  V' = Z - cV - \bar f
 \end{align}
we recall that the eigenvalues of the coefficient matrix are $-c\pm 1$, meaning that for $\vert c\vert <1$ the solution becomes unbounded as $\tau$ becomes large (positive or negative) for generic initial conditions.

\subsection{Velocity of polarization waves}
\label{ss:linearization}

We shall now find admissible velocities for (de)polarization waves by assuming the wave profiles to be monotone. To this end, we will find necessary conditions for the existence of monotone wave fronts for finite $a>0$ by linearizing the system \eqref{eq:TWfull} of ordinary differential equations around its equilibria and checking the eigenvalues of the linearization for being real. 

Before writing the traveling wave equations as first order system we now integrate the equation for $U$ once to find the first integral $aU'+cU-W$ and restrict the system to the invariant hyperplane $aZ+cU-W = C_1$. Keeping this time the constant of integration $C_1$ and replacing $Z = \frac1a(C_1-cU+W)$ yields 
\begin{equation} \label{eq:TWfull_system_reduced}
 \begin{array}{rllcl}
  U' = & -\frac{c}{a} U & +\frac{1}{a} W && + \frac{C_1}{a}  \\
  W' = & & & V & \\
  V' = &-\frac{c}{a^{2}} U & + \frac{1}{a^{2}}W & - \frac{c}{a} V & + \frac{C_1}{a^{2}} - \frac{1}{a} f_{0}(U,W)
 \end{array}
\end{equation}
We should note that the intersection of these invariant hyperplanes with the slow manifold $\mathcal{S}_a$ are just the perturbations of the leafs of the foliation we alluded to in preparation to Theorem \ref{thm:TWexist}. We therefore find the same types of heteroclinic orbits in each of these hyperplanes, and we will presently parametrize them by the $U$-value $U_1 = \frac{1}{c}C_1$ at their intersection with the $U$-axis.

Possible asymptotic states to be connected, or equilibria, are the same as for system \eqref{eq:TWfull_system} and are given by the zeroes of $f$. They shall be denoted by $(U_{1}, W_{1}\!=\!0)$, $(U_{2}, W_{2}\! = \!-U_{2})$, and $(U_{3}, W_{3}\! = \!U_{3})$, where we skipped the trivial variables $V\!
=\!Z\!=\!0$ and only kept the physical ones, $U$ and $W$. The relative values of the $U_i$ belonging to one heteroclinic orbit are fixed by the requirement that any two connected equilibria must belong to the same invariant hyperplane. The physical meaning behind this fact is simply mass conservation. A traveling wave front moving to the right at speed $c>1$ and connecting, say, a fully right polarized state $U\!=\!W\!=\!U_3$ to the non-polarized state $(U\!=\!U_1, W\!=\!0)$ swallows particles at rate $\frac{c-1}{2}U_1$ right oriented particles and $\frac{c+1}{2}U_1$ left oriented ones, together $cU_1$. At its rear, it leaves $(c-1)U_3$ right oriented ones behind. This balance determines the relation $U_3=\frac{c}{c+1}U_1$ for this particular type of waves.

Having restricted the system to a three dimensional hyperplane, we can now ask for the type of the three equilibria which, $c$ being different from $0$ and $\pm1$, certainly have stable and unstable manifolds of total dimension three. The linearization of \eqref{eq:TWfull_system_reduced} at the equilibrium point $(U,W,V) = (U_i,W_i,0)$ reads
\begin{equation} \label{eq:TWfull_system_reduced_lin}
 \begin{pmatrix}
  (U-U_i)' \\ (W-W_i)' \\ V' 
 \end{pmatrix}
=\begin{pmatrix}
  -\frac{c}{a} & \frac1a & 0 \\ 0 & 0 & 1 \\ -\frac{c}{a^2} - \frac1a f_U & \frac{1}{a^2} - \frac1a f_W & -\frac{c}{a}
 \end{pmatrix}
  \begin{pmatrix}
  U-U_i \\ W-W_i \\ V 
 \end{pmatrix}
\end{equation}
with the abbreviations $f_U=\partial_U f_0\vert_{(U_i,W_i)}$ and  $f_W = \partial_W f_0\vert_{(U_i,W_i)}$. The eigenvalues are the solutions of the characteristic equation (matrix scaled by a factor of $a$)
\begin{equation}\label{eq:char_eqn}
 \lambda(\lambda+c)^2 + (a f_W-1)(\lambda + c) + c+af_U = 0.
\end{equation}
The partial derivatives are calculated to be
\begin{equation}
 f_U  = \begin{cases}
         0 &~\mbox{ at } W= 0 \\ \pm 2 e^{-\beta^2 U^2} &~\mbox{ at } W = \pm U 
        \end{cases}
        ~~~~\mbox{ and }~~~~f_W = 
        \begin{cases}
         e^{-\beta^2U^2} &~\mbox{ at } W= 0 \\ - 2 e^{-\beta^2 U^2} &~\mbox{ at } W = \pm U 
        \end{cases}
\end{equation}

The fully polarized states $W\!=\!\pm U$ turn out to be saddle points for arbitrary values of the wave speed $c\notin \{-1,0,1\}$ and the parameter $a>0$. More precisely, the eigenvalues (scaled by a factor of $a^{-1}$) turn out to  be 
$$\lambda_1 = 1-c,~~~~~~\lambda_{2/3} = -\frac12\left( 1+c \pm \sqrt{(1+c)^2+8a\,e^{-\beta^2 U_2^2}} \right)$$
at $W_3=U_3>0$, and
$$\lambda_1 = -(1+c),~~~~~~\lambda_{2/3} = \frac12\left( 1-c \pm \sqrt{(1-c)^2+8a\,e^{-\beta^2 U_3^2}} \right)$$
at $W_2=-U_2<0$. In both cases, $\lambda_{2/3}$ are of opposite sign, and the sign of $\lambda_1$ is determined only by $c$. In particular, the eigenvalues of the linearization are real at the outer equilibria and oscillations are not to be expected. A necessary condition for the existence of monotone wave fronts connecting any of the fully polarized states and the non-polarized one can thus only be obtained from the linearization at the latter where nonreal eigenvalues are possible. 

\begin{SCfigure}[40][hbt]
 \includegraphics[width=60mm,height=45mm]{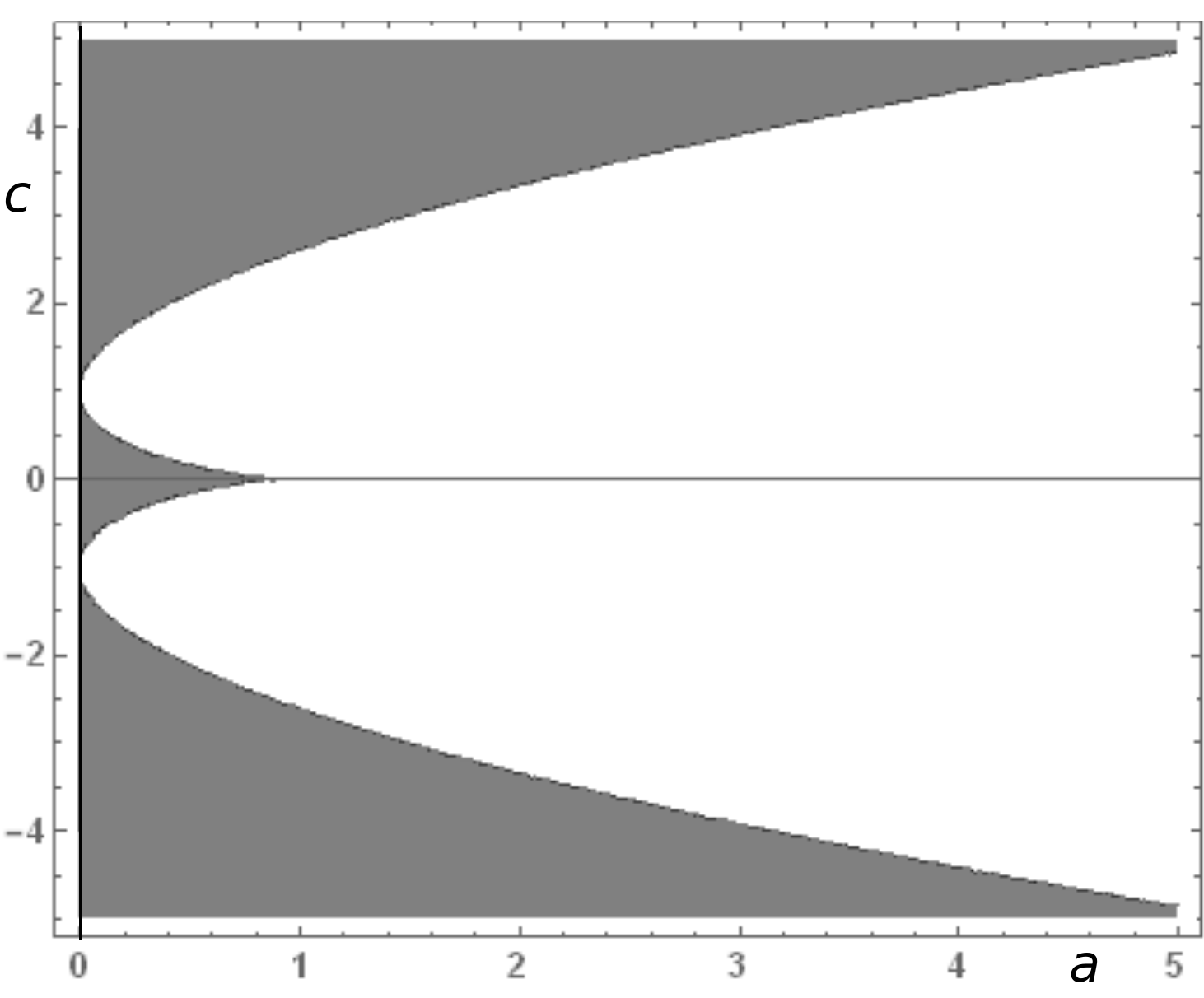}
 \caption{Behavior of the linearization about the non-polarized steady state $W=0$ depending on the parameter $a\,e^{-\beta^2 U_1^2}$ (abbreviated as $a$ in the figure) and the wave speed $c$. In the dark region, all eigenvalues are real, in the light region, there is a pair of complex conjugate eigenvalues.}
 \label{fig:cstar_full}
\end{SCfigure}

The characteristic equation \eqref{eq:char_eqn} at $W=0$ may be explicitely solved for the scaled eigenvalues $a\lambda$ but the expressions are rather long and hard to analyze. However, we can check for the eigenvalues being real by plotting the region where the sum $\vert\mathrm{Im}\lambda_{1}\vert+\vert\mathrm{Im}\lambda_{2}\vert+\vert\mathrm{Im}\lambda_{3}\vert$ is zero. The result is shown in Figure \ref{fig:cstar_full}, and we observe that for any given $c\neq \pm 1$ there is an $a_0(\vert c\vert)>0$ (depending also on the $U$-value of the point where we linearize) such that for any $a<a_0(\vert c\vert)$ the eigenvalues are real. 

From the singular perturbation argument we already knew that starting from the $c$-axis, except from $c=\pm 1$, and moving to the right in the diagram the behavior of the equilibria remains the same as for the purely hyperbolic limit system as long as $a$ is sufficiently small. Figuratively, only the hyperbolic directions of the fast flow are adjoined to the slow subsystem . Figure \ref{fig:bifurc} now illustrates what it means for $a$ to be small for given $c$, and it should not be surprising that the range of admissible $a$ increases as $c$ moves away from $\pm 1$.

\emph{Remark.} The dependence on $U_1$ arises from the crowding term in the alignment function $f_0$. If we neglect crowding by assuming a dilute regime ($U\ll\frac{1}{\beta}$), this term may be neglected and we only deal with the parameters $c$ and $a$. In any case, the dependence of $a_0$ on $U$ does not at all hurt the validity of the result since the exponential term is bounded by $1$.

\begin{figure}[hbt]
 \includegraphics[height=54mm,width=0.57\linewidth]{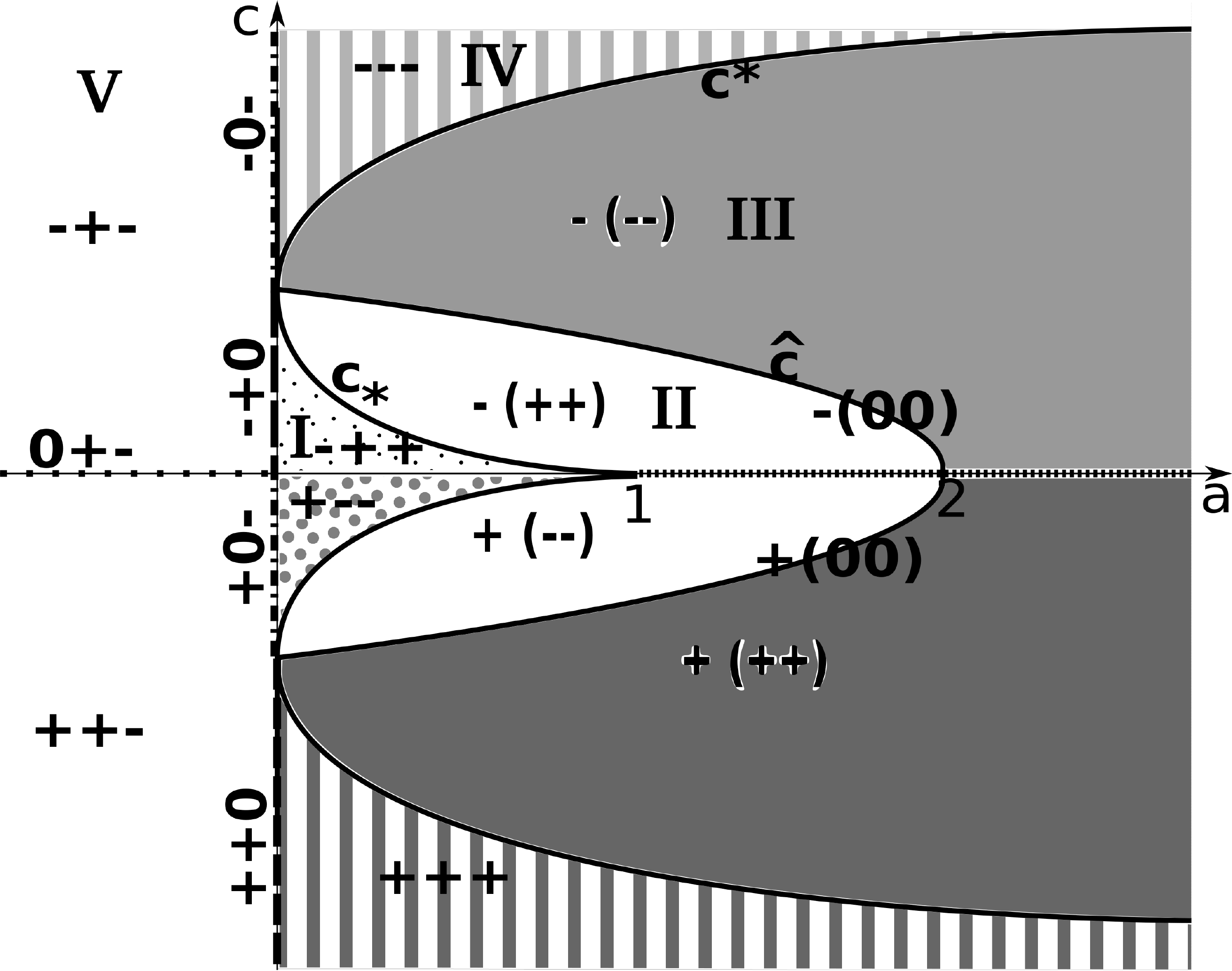}
 \caption{Bifurcation diagram for the ODE system \eqref{eq:TWfull_system_reduced}, with the type of equilibria sketched versus $a\,e^{-\beta^2U_1^2}$ and wave speed $c$. Indicated in each region are the signs of real parts of the eigenvalues of the linearization about the equilibria. Pairs of signs in parentheses indicate pairs of complex conjugate eigenvalues. In regions $\mathrm{II}$ and $\mathrm{III}$, oscillations and therefore no monotone wave fronts are to be expected. Clearly, we are only interested in the case $a\,e^{-\beta^2U_1^2}>0$
}
 \label{fig:bifurc}
\end{figure}

As everything is symmetric with respect to reflection at $x=0$ (which simply means interchanging left and right) we turn our attention to the upper region in Figure \ref{fig:bifurc} where the wave speed is positive. Let us now fix $a$ and $U_1$ and discuss the behavior depending on $c$. We find two critical velocities $c^*>1$ and $c_* < 1$, both depending on the parameter $a \,e^{-\beta^2U_1^2}$. For wave speeds $c\in(c_*(a), c^*(a))$ we cannot expect monotone wave fronts with the non-polarized equilibrium as asymptotic state since in that parameter region, trajectories spiral into this equilibrium (or out of it). However, for sufficiently (depending on $a$) large velocities $c$ we have a node which may be connected to any of the outer saddle points by a heteroclinic orbit corresponding to a monotone polarization wave.

Observe moreover, that for $a=2(1-c^2)>0$ (cf. the curve $\hat c(a)$ in Figure \ref{fig:bifurc}), there is a pair of purely imaginary eigenvalues $\pm (1-c^2)\imath = \pm \frac{a\imath}{2}$. This switch extends to the perturbed system ($a>0$) the stability change of the critical manifold at $c=1$.

Summarizing the results of this section, we conclude with the following prediction which types of polarization waves we might possibly expect to find in simulations of the full problem \eqref{eq:full}.

\noindent \textit{Prediction}. Given the constant of integration $C_1 = c\,U_1$ we find the following possible combinations of asymptotic states and expected wave fronts depending on the wave speed $c$:
\begin{enumerate}
  \item[$(i)$]{fast wave fronts with $c\geq c^* >1$ connecting either of the states 
  $$U_3\!=\!W_3\!=\!U_{1}c/(c-1) ~\mbox{ or }~ U_2\!=\!-W_2\!=\!U_{1}c/(c+1) ~~~\mbox{ to }~(U\!=\!U_{1}, W\!=\!0)$$ 
  if $c\geq c^{*}(a)$,}
  \item[$(ii)$]{slow wave fronts with $0<c\leq c_*<1$ connecting the state $$(U\!=\!U_{1}, W\!=\!0) ~~\mbox{ to  }~~ U_2\!=\!-W_2\!=\!U_{1}c/(c+1)$$ if and only if $c \leq c_{*}(a)$,}
 \end{enumerate}
and the mirror images of these fronts for negative velocities. 

We should remark that we give no proof of the stability of any of these waves as solutions of the PDE system \eqref{eq:full}. However,  we will see in section \ref{sec:simulations} that only the fast polarization waves are observed in simulations which leads to the conjecture that the slow depolarization waves are unstable.

\section{Inversion waves} \label{sec:inversion}

Having found that given the parameter $a$, monotone (de)- po\-lar\-iz\-at\-ion waves can only exist for a limited range of wave velocities, we may ask whether there might be other wave types with velocity $c_* < \vert c\vert < c^*$. Natural candidates are \emph{polarization inverting waves}, or \emph{inversion waves} for short, connecting two fully aligned states $W\!=\!U$ and $W\!=\!-U$ which amount to wave fronts running through the fully polarized medium and inverting its polarization from fully right aligned to fully left aligned or vice versa.

The existence of inversion waves connecting two totally aligned states could not be predicted by the perturbation theory employed in the previous section. 
However, we may illustrate their emergence by sketching the trajectories of the three dimensional system \eqref{eq:TWfull_system_reduced} for different wave velocities. For this sake we solve this system of ordinary differential equations for suitable initial conditions close to the saddle point $(U_3,W_3\!=\!U_3)$ and observe its evolution along the saddle's unstable manifold which for the leading fronts runs into the central node at $(U_1,0)$, where $U_1$ and $U_3$ satisfy relation $(c-1)U_3\! = \!cU_1$.

\begin{figure}[hbt]
 \includegraphics[width=0.49\linewidth, keepaspectratio]{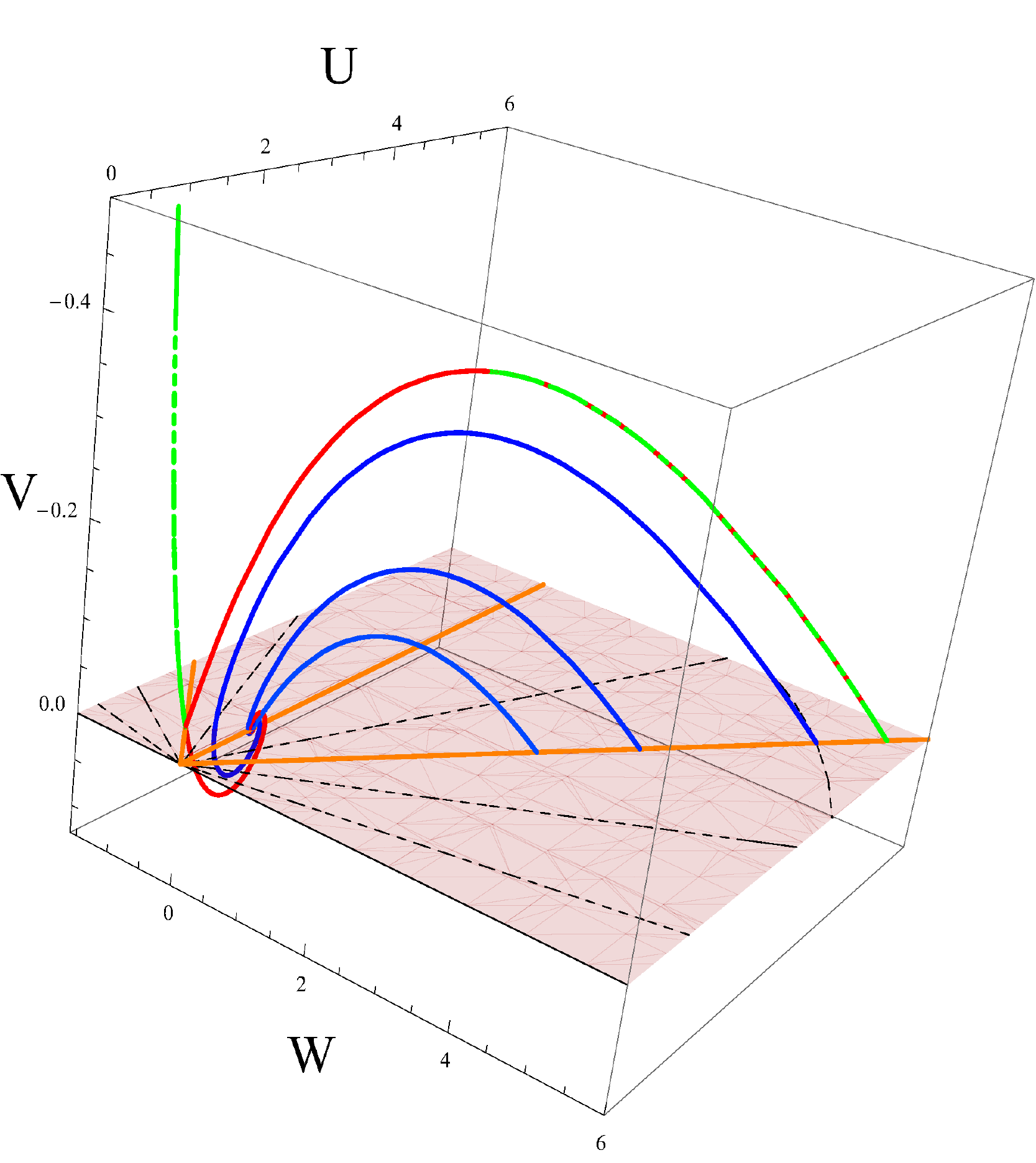}\hskip1mm %
  \includegraphics[width=0.49\linewidth, height=0.35\linewidth]{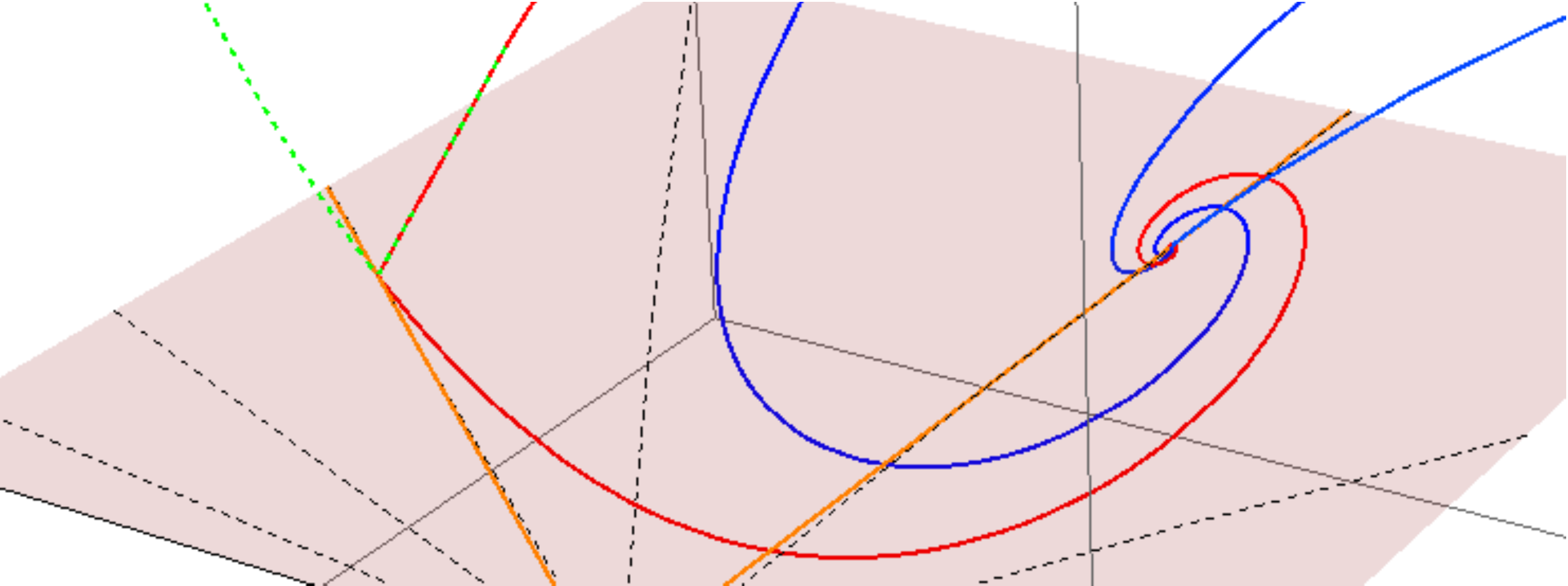}
  
\caption{\textit{Left}: Heteroclinic orbits starting at the state $(U_3,U_3)$ (on the right orange ray, selected according to the chosen speed $c$) entering into the central steady state $(U_1\!=\!1,0)$ or vanishing off to infinity (green trajectory). The heteroclinic orbit corresponding to an inversion wave lies between the hardly distinguishable green and red trajectory and cannot be shown as it is not generic. The velocities $c$ are (in ascending order of the starting point value $U_3$): $c=c^*\approx1.6,$\newline$c = 1.3,~c=1.25$, $c=1.2162466$ \newline$c=1.2162465$ (slightly smaller than speed of the inversion wave).\newline
\textit{Right}: zoom to central and left equilibrium. Notice the sharp distinction between the red and the green trajectory.}
\label{fig:trajectories}
\end{figure}

In Figure \ref{fig:trajectories} we see trajectories corresponding to monotone wave fronts for $c\geq c^*(a)$ as predicted from the linearization. Decreasing the velocity, the orbits start spiraling into the central steady state $(U_1,0)$ although this spiraling is barely visible for $c$ only slightly smaller than $c^*$. 

Decreasing $c$ further towards $1$ leads to increasing oscillations until the trajectories become so large as to hit the opposite equilibrium $(U_2,W_2\!=\!-U_2)$ and do not return towards the central equilibrium. Now the (one dimensional) unstable manifold of the first saddle point lies in  the (two dimensional) stable manifold of the second one and we obtain a heteroclinic orbit connecting these two saddle points. That happens for precisely one wave speed $c>1$ which we cannot determine analytically but observe as the speed of the inversion waves in the simulations. 

A further decrease of the wave speed makes the unstable manifold of the first saddle point lie on the other side of the stable manifold of the second one, and the trajectory now approaches this second point $(U_2,W_2\!=\!-U_2)$ but then quickly vanishes off to infinity, which obviously does not make sense as a solution of the system of partial differential equations. That means that we should not expect wave front solutions for these smaller velocities.

\section{Wave patterns emerging in simulations}
\label{sec:simulations}

So far we have discussed the possible monotone wave patterns without knowing which of them are stable and can be expected to be observed as asymptotic solutions of problem \eqref{eq:full} for generic initial conditions. We therefore try to find traveling wave solutions in simulations of the full model \eqref{eq:full} on a large domain. 

\subsection{Types of wave observed in simulations}

Let us consider two types of initial data which both turn out to give rise to solutions consisting of different wave fronts traveling at specific speeds. Typical examples for the data and resulting wave patterns are shown in Figure \ref{fig:wave_shapes}. The initial data are constructed by locally removing or adding some particles compared to the homogeneous non-polarized state $u\equiv u_0$, $w\equiv0$. It should be noted that the spatial scale for the initial data is strongly zoomed, meaning that the total initial deviation from the homogeneous state is small. Still, the pronounced wave patterns seen in the right panel of \ref{fig:wave_shapes} emerge from these localized inhomogeneities. Since in the resulting patterns the leading wave fronts are faster than the trailing ones, the peaks get wider as time proceeds. 

\begin{figure}[hbt]
 \includegraphics[width=0.99\linewidth,height=62mm]{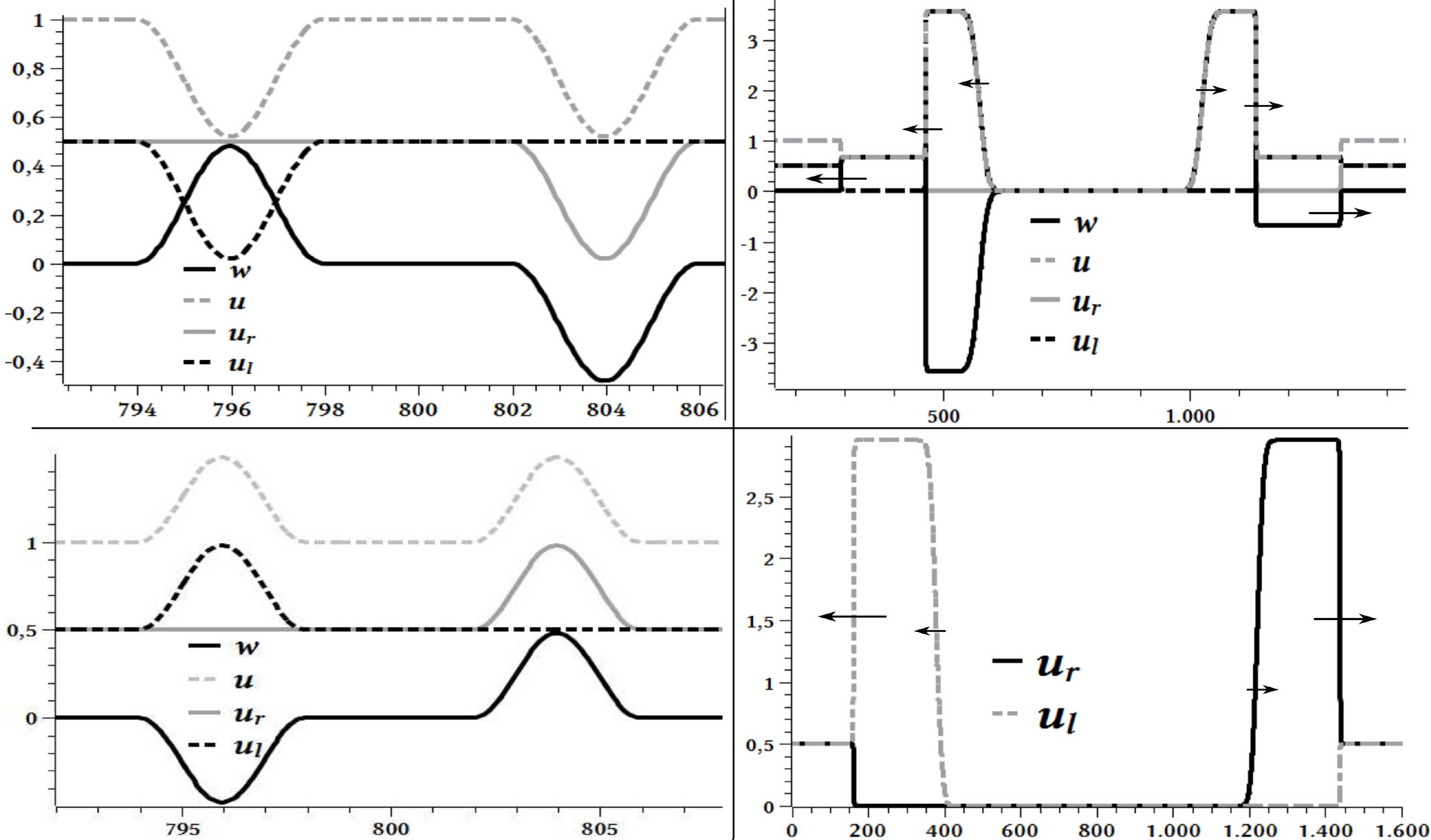}
 \caption{Typical patterns (\textit{right}) emerging from symmetric initial conditions (\textit{left}). Removing some particles of either orientation from the non-polarized state (\textit{top}) results in solutions with two wave fronts followed by one diffusion front for each direction. Adding some particles (\textit{bottom}) results in a single wave front followed by a diffusion front in each direction. The arrows designate the velocity of the respective fronts.}
\label{fig:wave_shapes}
\end{figure}

The next observation to report is that we indeed find some of the wave fronts predicted in the previous sections. Most prominently, depending on the initial conditions we find different types of fast polarization waves with velocities $\vert c\vert >1$, namely those connecting any of the two fully polarized states $W=\pm U$ to the non-polarized state $W=0$. Moreover, the precise relation between the plateau values $U_i$ at the asymptotic states and the wave speeds $c$ are as predicted at the end of subsection \ref{ss:linearization}. We recall that this relation can also be inferred from mass conservation for the particles.

\begin{figure}[hbt]
\includegraphics[width=65mm, height=68mm]{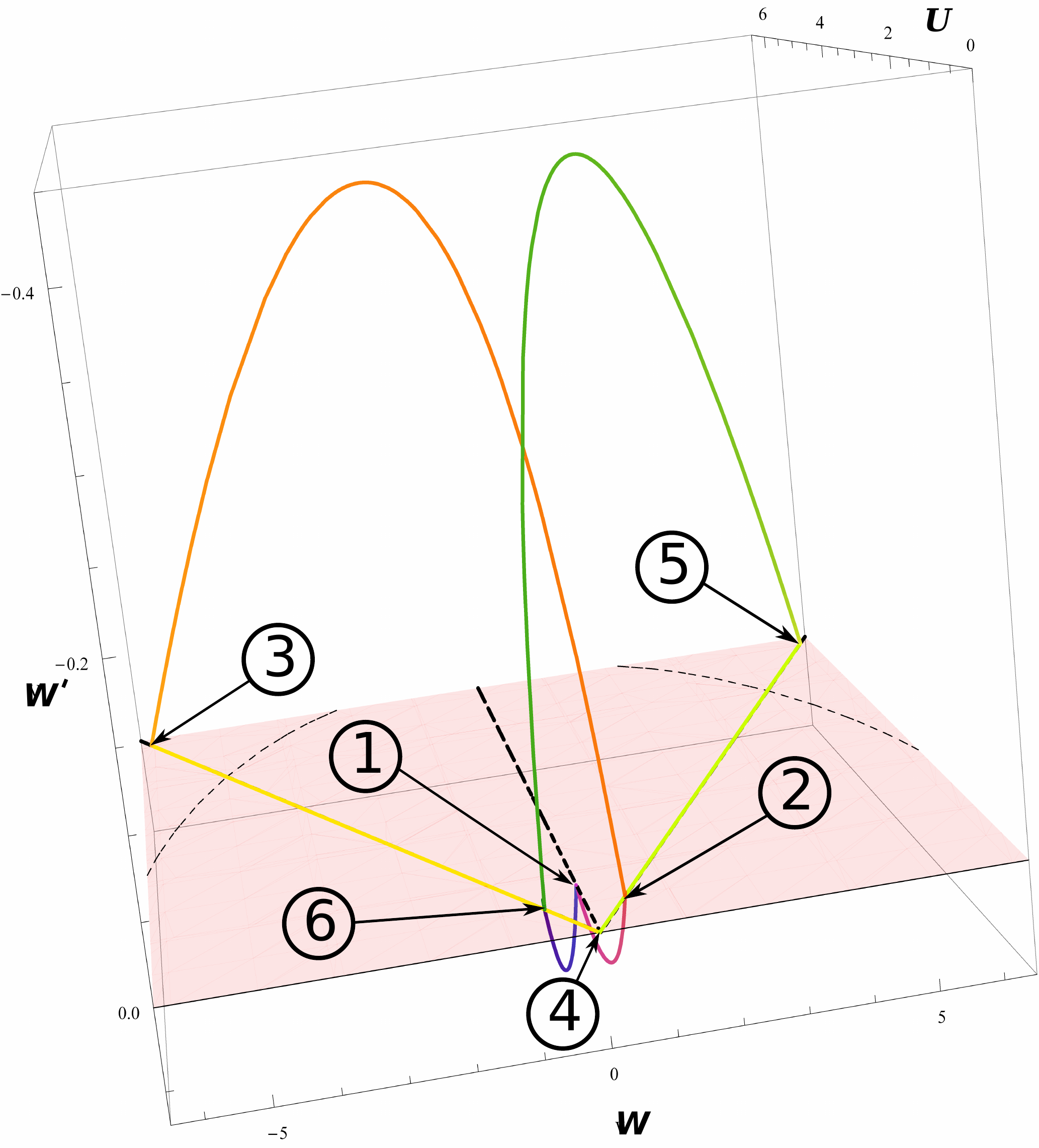}
\caption{Illustration of the $U$-$W$-$W'$ phase space trajectories corresponding to a wave pattern as emerging in the top panel of Figure \ref{fig:wave_shapes}. Corresponding to the plateau values in Figure \ref{fig:wave_shapes} from left to right are the states $1$ (asymptotic state $U=W=1$), $2$ (small hump, only right moving particles), $3$ (large hump, fully polarized to the left), $4$ (central region devoid of particles), $5$ (large hump, only right moving particles), $6$ (small hump, fully left polarized), and $7$ (asymptotic non-polarized state again). The diffusion profiles between the states $3,4,$ and $5$ are shown as yellow straight lines on the rays $U = \pm W, W'=0$ (projected into the $W'=0$ plane). The purple and blue trajectories correspond to polarization waves whereas inversion waves are depicted in orange and green.}
 \label{fig:wave_pattern_phase}
\end{figure}

In addition, we also observe inversion waves connecting the two totally aligned states to one another. These are the second edges of the wave patterns in the two-front solutions in the top panel of Figure \ref{fig:wave_shapes}.

Moreover, we observe additional patterns which we shall call \emph{diffusion fronts}. 
They connect the trivial state $W\!=\!U\!=\!0$ emerging in the center of the domain with one of the totally aligned states. It shall be noted that they are just solutions of the pure diffusion-advection equation 
$$\partial_t u_{r/l} \pm \partial_{x}u_{r/l} \!=\! a \partial_{xx}u_{r/l}$$ 
since the alignment term vanishes identically along these profiles. In the $U$-$W$ plane, these diffusion profiles simply correspond to orbits along the lines $W=\pm U$.

The orbits in the phase space of \eqref{eq:TWfull_system_reduced} corresponding to the wave pattern in the top panel of Figure \ref{fig:wave_shapes} are sketched in Figure \ref{fig:wave_pattern_phase}. Note that the diffusion profiles at the rear of the outward moving wave trains in Figure \ref{fig:wave_shapes} are represented by the straight lines connecting the outermost points with the origin in Figure \ref{fig:wave_pattern_phase} while the entire depletion zone in the center is captured by the origin.

\subsection{Observed wave velocities depending on the parameters}

Concerning the observed velocities we note that we find as polarization waves precisely those with critical velocity $c^{*}$ as determined by the linearization, to be seen in Figure \ref{fig:wavespeeds}. This behavior is known from simple reaction-diffusion equations as the Fisher-KPP equation where precisely those wave fronts with critical velocity are shown to be stable (cf. \cite{VolPet2009}). In particular, the velocity of these waves indeed only depends on the product $a=\alpha\varepsilon$ rather than on both parameters individually. 

The slower inversion waves arising in the two-front patterns have a velocity $\tilde c\in (1,c^{*})$. In that case, the relation between wave speed and value of the asymptotic states $U_{2}=W_{2}$ and $U_{3}=-W_{3}$ reads
\begin{equation}
 U_{3} = \frac{\tilde c-1}{\tilde c+1} U_{2}.
\end{equation}
As with $c^*$, the velocity $\tilde c$ only depends on the product $a$ and can be accurately predicted by simulations of system \eqref{eq:TWfull} of ordinary differential equations. We obtain $\tilde c$ as that velocity at which the trajectories in \ref{fig:trajectories} switch from spiraling into the central focus to vanishing of towards infinity. The values shown in Figure \ref{fig:wavespeeds} are obtained by this method which yields more accurate values than the simulation of the PDE system and subsequent measurement of the velocity as displacement of the front divided by the time elapsed.

\begin{figure}[hbt]
 \includegraphics[width=0.45\linewidth, height=36mm]{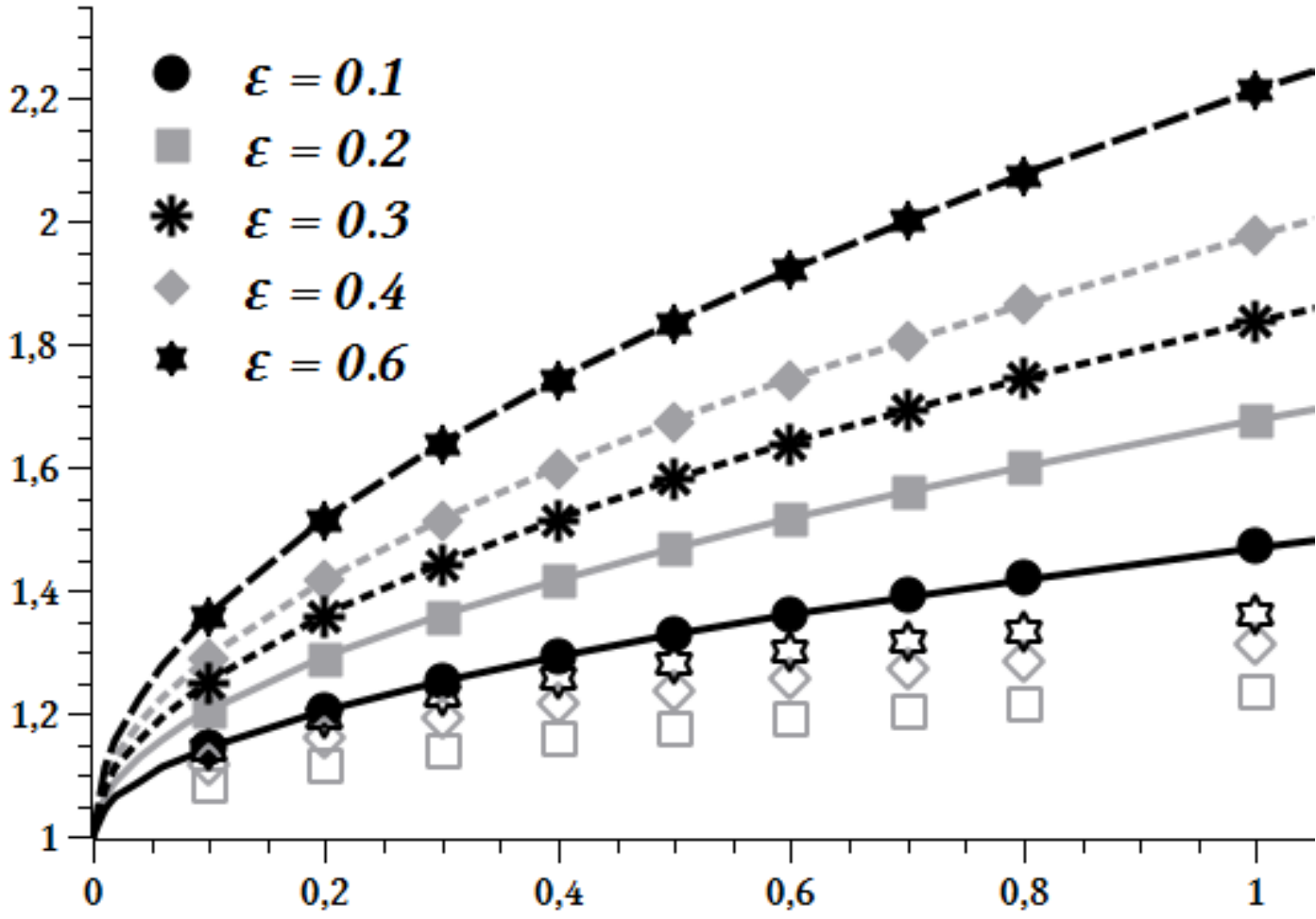} \hskip0.01\linewidth %
 \includegraphics[width=0.53\linewidth, height=36mm]{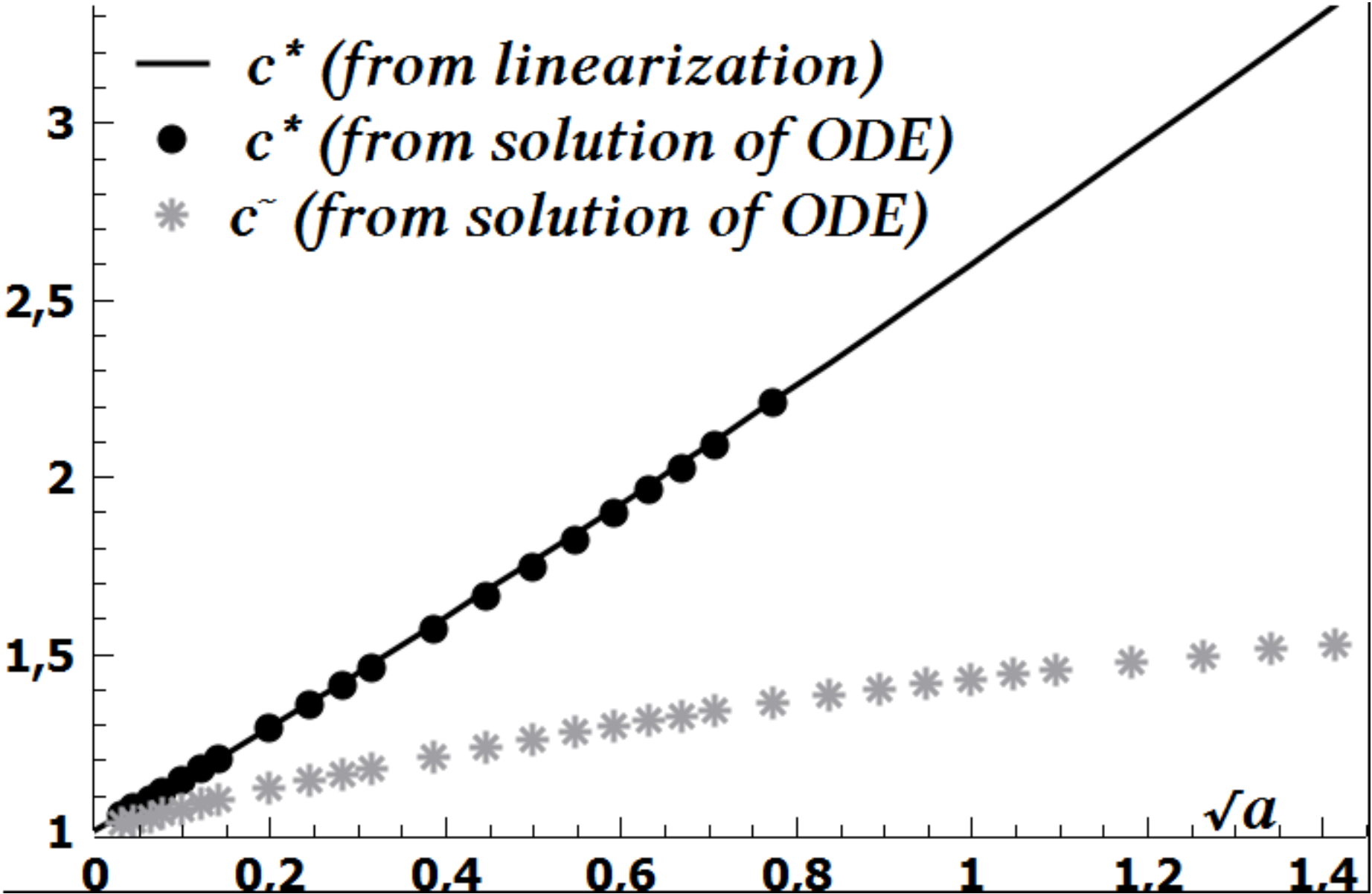}  
 \caption{Observed wave speeds in simulations at the leading polarization (\textit{full symbols}) and trailing inversion fronts (\textit{hollow symbols}) for two-front solutions depending on $\alpha$ for different values of $\varepsilon$ (\textit{left}) without exponential crowding ($\beta=0$). Note the perfect agreement of the measured velocities with $c^{*}$ (\textit{lines}) predicted from the linearization. \textit{Right}: Values of $c^*$ obtained from the linearization (\textit{solid line}) and from the simulation of system \eqref{eq:TWfull}(\textit{circles}) and velocities of the inversion waves ($\tilde{c}$, \textit{stars}) computed from simulations of system \eqref{eq:TWfull} plotted against $\sqrt{\alpha\varepsilon}$.}
 \label{fig:wavespeeds}
\end{figure}

For the wave speeds of the single-front solutions sketched in Figure \ref{fig:wave_shapes} we observe precisely the same values as for the leading fronts of two-front patterns. This fits perfectly well with the theory since the linearization leading to the prediction in subsection \ref{ss:linearization} does not distinguish between these patterns. 

Finally, the velocities for the trailing fronts always lie between the particle velocity $1$ and the critical velocity $c^{*}$. They are also monotonically increasing functions of the product $\alpha\varepsilon$ 
 but do not satisfy any obvious relationship with $c^{*}$.

\section{Conclusion and discussion}\label{sec:conclusion}

Considering a model for the directed flow of particles undergoing diffusion and mutual alignment that was motivated by a model for actin filaments in the cytoskeleton of motile cells we found several types of traveling wave solutions. Starting from the easily accessible hyperbolic limit system without diffusion we analytically showed the existence of wave fronts moving at any non-zero velocity different from the particle velocity if the diffusion coefficient is sufficiently small. 

Numerical simulations showed that the wave fronts moving at some critical velocity $c^{*}$ determined from the linearization of the system do really emerge from rather arbitrary initial data and thus appear to be stable. Moreover, we found another type of wave fronts traveling at smaller velocities which were neither predicted by the linearization of the equations nor by the singular perturbation theory. However, a closer examination of the system of ordinary differential equations describing the wave shapes showed how these fronts emerge as trajectories connecting two saddle points.  

Of particular interest are those wave patterns emerging from initial perturbations of the homogeneous non-polarized state with some additional particles added for either direction. The resulting solution was comprised of two humps of particles moving outward from the center at the critical velocity specified above, growing in width so that their rear flank moves at the particle velocity $1$, and leaving behind a central region devoid of any particles. This emergence of multiple wave fronts of different velocities has been known for a long time (e.g., \cite{ShkMatVol1993}) but it is remarkable that complicated patterns of this type are observed in a rather minimal system of two equations.

We now return to the motivation for the proposed model -- the investigation of a system of equations describing the movement of polymerizing and depolymerizing actin filaments. The traveling wave solutions we found may be interpreted as polymerization fronts as observed in, e.g., \cite{Vic2000}. As traveling waves describing a persistently moving cell are among the most interesting solutions regarding cell motility (e.g., \cite{RecPutTru2013, BarLeeAllTheMog15}), this strongly suggests further investigation of possible alignment mechansims to be incorporated in the original cytoskeleton model in \cite{FuhKasSte2007}. The shock-like solutions observed for this model in \cite{FuhSte2015} were not easily understood in the context of the hyperbolic-parabolic system. However, in terms of a more comprehensive model it may be expected to find them explained as reminiscent of the traveling wave solutions found here, if we recall that the filaments in the original model were not assumed to be short and therefore, their diffusion was neglected.

\section*{Acknowledgments}
The work of J. Fuhrmann was partially supported by the German Ministry of Education and Research (grants 01GQ1003A, 01GS08154 \cite{NGFN2010}).

\bibliographystyle{plain}
\bibliography{references}
\end{document}